\documentclass[english]{article}
\usepackage[T1]{fontenc}
\usepackage[latin9]{inputenc}
\usepackage{amsthm}
\usepackage{amsmath}
\usepackage{amssymb}
\usepackage{graphicx}
\usepackage{esint}

\makeatletter
\numberwithin{equation}{section}
  \theoremstyle{remark}
  \newtheorem{rem}{\protect\remarkname}[section]
  \theoremstyle{definition}
  \newtheorem{defn}{\protect\definitionname}[section]
  \theoremstyle{plain}
  \newtheorem{thm}{\protect\theoremname}[section]
  \theoremstyle{plain}
  \newtheorem{lem}{\protect\lemmaname}[section]
  \theoremstyle{plain}
  \newtheorem{prop}{\protect\propositionname}[section]

\makeatother

\usepackage{babel}
  \providecommand{\definitionname}{Definition}
  \providecommand{\lemmaname}{Lemma}
  \providecommand{\propositionname}{Proposition}
  \providecommand{\remarkname}{Remark}
\providecommand{\theoremname}{Theorem}

\begin{document}

\title{On an Inversion Theorem for Stratonovich's Signatures of Multidimensional
Diffusion Paths}

\author{X. Geng%
\thanks{Mathematical Institute, University of Oxford, Oxford OX2 6GG, England
and the Oxford-Man Institute, University of Oxford, Oxford OX2 6ED,
England. \protect \\
Email: xi.geng@maths.ox.ac.uk. %
}\  and Z. Qian%
\thanks{Exeter College, University of Oxford, Oxford OX1 3DP, England. \protect \\
Email: qianz@maths.ox.ac.uk.%
}}
\date{}
\maketitle
\begin{abstract}
In the present paper, we prove that with probability one, the Stratonovich
sigatures of a multidimensional diffusion process (possibly degenerate)
over $[0,1],$ which is the collection of all iterated Stratonovich's
integrals of the diffusion process over $[0,1],$ determine the diffusion
sample paths. 
\end{abstract}


\section{Introduction }

Let $X_{t}$ be an $\mathbb{R}^{d}$-valued continuous path over $[0,1]$
with bounded variation ($d\geqslant2$). According to \cite{lyons2007differential},
\cite{lyons2003system}, for $0\leqslant s<t\leqslant1,$ we can define
the sequence of iterated integrals 
\[
\mathbf{X}_{s,t}=(1,X_{s,t}^{1},X_{s,t}^{2},\cdots,X_{s,t}^{n},\cdots),
\]
where

\begin{equation}
X_{s,t}^{n}=\int_{s<u_{1}<\cdots<u_{n}<t}dX_{u_{1}}\otimes\cdots\otimes dX_{u_{n}},\ \ \ n\geqslant1.\label{iterated integrals}
\end{equation}
$X_{s,t}^{n}$ is regarded as an element in the tensor space $(\mathbb{R}^{d})^{\otimes n}\cong\mathbb{R}^{nd}$
and $\mathbf{X}_{s,t}$ is hence an element in the tensor algebra
\[
T^{(\infty)}(\mathbb{R}^{d})=\oplus_{n=0}^{\infty}\mathbb{R}^{nd}.
\]
$\mathbf{X}_{s,t}$ is multiplicative in the sense that it satisfies
the following Chen's identity:
\[
\mathbf{X}_{s,t}=\mathbf{X}_{s,u}\otimes\mathbf{X}_{u,t},\ \ \ 0\leqslant s<u<t\leqslant1.
\]
$\mathbf{X}_{s,t}$ is uniquely determined by the original path $X_{t};$
or intuitively speaking, the original path $X_{t}$ contains all information
about its differential $dX_{t}$. A remarkable consequence is that
a theory of integration along $X_{t}$ can be established in the sense
of Riemann\textendash{}Stieltjes, which leads to a theory of differential
equations driven by $X_{t}.$ Such a theory for paths with bounded
variation is classical and well-studied. 

If the path $X_{t}$ is less regular, for example, $X_{t}$ has finite
$p$-variation for some $p>1,$ it may not be possible to establish
an integration theory along $X_{t}$ by using the information of the
original path only. The fundamental reason is that the path $X_{t}$
itself does not reveal enough information on its differential $dX_{t}$
, which is essential to be fully understood if we want to develop
an integration theory along $X_{t}.$ As pointed out by T. Lyons in
\cite{lyons1998differential}, for this purpose, together with the
path itself, a finite sequence of iterated integrals up to level $[p]$
satisfying Chen's identity should be specified in advance. Such a
finite sequence of iterated integrals
\[
\mathbf{X}_{s,t}=(1,X_{s,t}^{1},\cdots,X_{s,t}^{[p]})
\]
is regarded as a multiplicative functional $\mathbf{X}$ from the
simplex $\Delta=\{(s,t):\ 0\leqslant s\leqslant t\leqslant1\}$ to
the truncated tensor algebra 
\[
T^{([p])}(\mathbb{R}^{d})=\oplus_{n=0}^{[p]}\mathbb{R}^{nd}.
\]
$\mathbf{X}$ is called a rough path with roughness $p.$ According
to \cite{lyons1998differential}, $\mathbf{X}$ extends uniquely to
a multiplicative functional from $\Delta$ to $T^{(\infty)}(\mathbb{R}^{d}).$
In the founding work of T. Lyons in \cite{lyons1998differential},
a general theory of integration and differential equations for rough
paths was established. 

For a rough path $\mathbf{X}$ with roughness $p,$ the signature
of $\mathbf{X}$ is defined as the formal sequence 
\[
S(\mathbf{X})=\mathbf{X}_{0,1}=(1,X_{0,1}^{1},\cdots,X_{0,1}^{[p]},\cdots),
\]
where for $n>[p],$ $X_{s,t}^{n}$ is the unique extension of $\mathbf{X}$
as mentioned before. The signature $S(\mathbf{X})$, proposed by K.T.
Chen in \cite{chen1977iterated} and T. Lyons in \cite{lyons1998differential},
can be regarded as the collection of overall information of any arbitrary
level $n$ about the rough path $\mathbf{X}.$ It is of central interest
and conjectured in the theory of rough paths that the signature $S(\mathbf{X})$
contains sufficient information to recover the path $\mathbf{X}$
completely. In the groundbreaking paper \cite{hambly2010uniqueness}
by B. Hambly and T. Lyons, they proved that for a path $X_{t}$ with
bounded variation, the signature of $X_{t}$ uniquely determines the
path up to a tree-like equivalence. However, for paths with unbounded
variation, very few results are available and it remains a lot of
work to do. 

In the work \cite{lejan2012stratonovich} by Y. Le Jan and Z. Qian,
they considered the case of multidimensional Brownian motion and proved
that for almost surely, the Brownian paths can be recovered by using
the so-called Stratonovich's signature, which is defined via iterated
Stratonovich's integrals of arbitrary orders. Since the Brownian paths
are of unbounded variation and can be regarded as rough paths with
roughness $p\in(2,3)$, we may need to specify the second level in
order to make sense in terms of rough paths. However, according to
\cite{ikeda1989stochastic}, \cite{wong1966relationship}, there is
a canonical lifting of the Brownian paths to the second level by using
dyadic approximations, which is called the L$\acute{\mbox{e}}$vy's
stochastic area process and it coincides exactly with the iterated
Stratonovich's integral defined in the same way as (\ref{iterated integrals}).
Such lifting is determined by the Brownian paths itself, and in \cite{lejan2012stratonovich}
when regarding the Brownian motion as rough paths such lifting was
used by the authors. Therefore, the recovery of Brownian motion as
rough paths is essentially the recovery of the Brownian paths in terms
of Stratonovich's signature.

In the present paper, we are going to generalize the result of Y.
Le Jan and Z. Qian in \cite{lejan2012stratonovich} to the case of
multidimensional diffusion processes (possibly degenerate). The main
idea of the proof is similar to the case of Brownian motion, in which
the authors used a specially designed approximation scheme and chose
special differential 1-forms to define the so-called extended Stratonovich's
signatures to recover the Brownian paths. However, there are several
difficulties in the case of diffusion processes. Firstly, we need
quantitative estimates for rare events of diffusion processes to prove
a convergence result similar to the case of Brownian motion. In \cite{lejan2012stratonovich},
the authors used the symmetry and explicit distribution of Brownian
motion, which are not available in the case of diffusion processes
and hence we need to proceed in a different way. Secondly, to construct
special differential $1$-forms, a quite special case of H$\ddot{\mbox{o}}$rmander's
theorem was used to ensure the existence of density, in which the
so-called H$\ddot{\mbox{o}}$rmander's condition was easily verified.
In the case of diffusion processes, the construction of differential
$1$-forms is more complicated to ensure similar kind of hypoellipticity.
Lastly, in the Brownian motion case, the Laplace operator is well-posed
so that PDE methods could be applied to obtain a crucial estimate
which enables us to relate the extended Stratonovich's signatures
to the Brownian paths. However, for a general diffusion process, the
generator $L$ may not be well-posed any more (we don't impose uniform
ellipticity assumption on $L$) and PDE methods may no longer apply
(in fact, to ensure the application of PDE methods, rather technical
assumptions should be imposed on the differential operator $L$ and
the domain if without uniform ellipticity). Therefore, we need a different
approach to recover the diffusion paths by using extended Stratonovich's
signatures.

\section{Main result and idea of the proof }

In this section, we are going to state our main result and illustrate
the idea of the proof.

Let $(\Omega,\mathcal{F},P)$ be a complete probability space and
let $W_{t}$ be a $d$-dimensional Brownian motion on $\Omega.$ Consider
an $N$-dimensional $(N\geqslant2)$ diffusion process $X_{t}$ defined
by the following SDE (possibly degenerate): 
\begin{equation}
dX_{t}=\sum_{\alpha=1}^{d}V_{\alpha}(X_{t})\circ dW_{t}^{\alpha}+V_{0}(X_{t})dt\label{SDE}
\end{equation}
with $X_{0}=0.$ 

We are going to make the following three assumptions on the generating
vector fields $\{V_{1},\cdots,V_{d};V_{0}\}$.

(A) $V_{0},V_{1},\cdots,V_{d}\in C_{b}^{\infty}(\mathbb{R}^{N}).$

(B) For any $x\in\mathbb{R}^{N},$ H$\ddot{\mbox{o}}$rmander's condition
(see \cite{hormander1967hypoelliptic}) holds at $x$ in the sense
that 
\[
V_{1},\cdots,V_{d},\ [V_{\alpha},V_{\beta}],0\leqslant\alpha,\beta\leqslant d,\ [V_{\alpha},[V_{\beta},V_{\gamma}]],0\leqslant\alpha,\beta,\gamma\leqslant d,\ \cdots
\]
generate the tangent space $T_{x}\mathbb{R}^{N}\cong\mathbb{R}^{N}$,
where $[\cdot,\cdot]$ denotes the Lie bracket. 

(C) There exists a positive orthonormal basis $\{e_{1},\cdots,e_{N}\}$
of $\mathbb{R}^{N},$ such that for any $x\in\mathbb{R}^{N}$ and
$i=1,2,\cdots,N,$ $V_{\alpha}(x)$ is not perpendicular to $e_{i}$
for some $\alpha=1,2,\cdots,d.$

\begin{rem}
Assumptions (A) and (B) are made to ensure the hypoellipticity of
the generator 
\[
L=\frac{1}{2}\sum_{\alpha=1}^{d}V_{\alpha}^{2}+V_{0}
\]
of the diffusion process (\ref{SDE}). Assumption (C) is made to ensure
the escape condition and the non-tangential condition proposed in
\cite{ben1984poisson} hold on some domain of interest which is relatively
small. Under these assumptions, we are able to apply results in \cite{ben1984poisson}
to obtain the existence of a continuous density function of the Poisson
kernel for some domain of interest and a quantitative estimate on
the density function, which are both crucial in the proof of our main
result. 

It should be pointed out that if the diffusion process (\ref{SDE})
is nondegenerate, that is, if $\{V_{1}(x),\cdots,V_{d}(x)\}$ generate
the tangent space $T_{x}\mathbb{R}^{N}\cong\mathbb{R}^{N}$ at each
point $x\in\mathbb{R}^{N},$ then Assumptions (A), (B), (C) are all
verified. 
\end{rem}

For $n\geqslant1,$ $j_{1},\cdots,j_{n}\in\{1,2,\cdots,N\},$ define
the iterated Stratonovich's integral of order $n$: 
\[
[j_{1},\cdots,j_{n}]_{s,t}=\int_{s<t_{1}<\cdots<t_{n}<t}\circ dX_{t_{1}}^{j_{1}}\circ dX_{t_{2}}^{j_{2}}\circ\cdots\circ dX_{t_{n}}^{j_{n}},\ \ \ 0\leqslant s<t\leqslant1.
\]
Alternatively, $[j_{1},\cdots,j_{n}]_{s,t}$ can be defined inductively
by the following relation:
\[
[j_{1},\cdots,j_{n}]_{s,t}=\int_{s<u<t}[j_{1},\cdots,j_{n-1}]_{s,u}\circ dX_{u}^{j_{n}},\ \ \ 0\leqslant s<t\leqslant1,
\]
where $[j_{1}]_{s,t}$ is defined to be 
\[
[j_{1}]_{s,t}=\int_{s<u<t}\circ dX_{u}^{j_{1}}=X_{t}^{j_{1}}-X_{s}^{j_{1}},\ \ \ 0\leqslant s<t\leqslant1.
\]
For convenience, if $n=0,$ we denote $[j_{1},\cdots,j_{n}]_{s,t}=1.$
The family 
\[
\{[j_{1},\cdots,j_{n}]_{0,1}:\ j_{1},\cdots,j_{n}\in\{1,2,\cdots,N\},\ n\geqslant0\}
\]
of iterated Stratonovich's integrals is called the \textit{Stratonovich
signature} of $X_{t}$ over $[0,1].$

Let $\mathcal{F}_{1}$ be the completion of the $\sigma$-algebra
generated by the diffusion process $X_{t}$ over $[0,1],$ and let
$\mathcal{G}_{1}$ be the completion of the $\sigma$-algebra generated
by the Stratonovich's signature of $X_{t}$ over $[0,1].$ More precisely,
\begin{eqnarray*}
\mathcal{F}_{1} & = & \overline{\sigma(X_{t}:\ 0\leqslant t\leqslant1)},\\
\mathcal{G}_{1} & = & \overline{\sigma(\{[j_{1},\cdots,j_{n}]_{0,1}:\ j_{1},\cdots,j_{n}\in\{1,2,\cdots,N\},n\geqslant0\})}.
\end{eqnarray*}
For the case of Brownian motion, it was proved by Y. Le Jan and Z.
Qian in \cite{lejan2012stratonovich} that 
\[
\mathcal{F}_{1}=\mathcal{G}_{1}.
\]
Such result for diffusion processes in our setting can be proved in
the present paper. However, we are going to formulate the problem
in a more illustrative way, which to some extend reveals how we can
reconstruct the diffusion paths from the Stratonovich's signature
over $[0,1]$ in a conceivable way.

First we need the following definition.
\begin{defn}
A \textit{piecewise linear trajectory} (P.L.T.) $\mathcal{T}$ in
$\mathbb{R}^{N}$ is a finite sequence of points in $\mathbb{R}^{N}$
(not necessarily all distinct). Here we always assume that the number
of points in $\mathcal{T}$ is greater than one (if $\mathcal{T}$
consists of only one point $x,$ we will regard $\mathcal{T}$as the
finite sequence $(x,x)$). For a P.L.T. $\mathcal{T}$ in $\mathbb{R}^{N}$,
the number of points in $\mathcal{T}$will be denoted by $|\mathcal{T}|.$
If the points of $\mathcal{T}$ belongs to a subset $\Gamma\subset\mathbb{R}^{N},$
we say that $\mathcal{T}$ is a P.L.T. in $\Gamma.$
\end{defn}

The reason why we use the notion ''piecewise linear trajectory'' is
that when given $\mathcal{T},$ we actually think of $\mathcal{T}$
as a piecewise linear graph by connecting the points in $\mathcal{T}$
by line segments in order. Here we should point out that the order
of points in $\mathcal{T}$ is rather important, and no parametrizations
are involved.
\begin{defn}
For $n\geqslant2,$ a \textit{parametrization} $\sigma$ of order $n$ is
a partition of the time interval $[0,1]$ into $n-1$ nontrivial subintervals:
\[
\sigma:\ 0=t_{1}<t_{2}<\cdots<t_{n-1}<t_{n}=1.
\]
The space of all parametrizations of order $n$ will be denoted by
$\mathcal{P}_{n}.$ 

Let $\mathcal{T}$ be a P.L.T. in $\mathbb{R}^{N}$ and let $\sigma$
be a parametrization of order $|\mathcal{T}|.$ The piecewise linear
path over $[0,1]$ defined by applying linear interpolation of $\mathcal{T}$
along the parametrization $\sigma$ is denoted by $\mathcal{T}(t|\sigma).$
\end{defn}

Our formulation of the problem is related to a kind of convergence
which is parametrization free. Therefore, we need the following definition
of convergence in trajectory.
\begin{defn}
Let $(\gamma_{t})_{0\leqslant t\leqslant1}$ be a continuous path
in $\mathbb{R}^{N}.$ A sequence $\{\mathcal{T}^{(n)}\}$ of P.L.T.s
is said to be \textit{converging in trajectory} to $(\gamma_{t})_{0\leqslant t\leqslant1}$
if 
\[
\lim_{n\rightarrow\infty}\inf_{\sigma\in\mathcal{P}_{|\mathcal{T}^{(n)}|}}\sup_{0\leqslant t\leqslant1}|\gamma_{t}-\mathcal{T}^{(n)}(t|\sigma)|=0.
\]

\end{defn}

\begin{rem}
Such kind of convergence modulo parametrization is similar to the
notion of Fr$\acute{\mbox{e}}$chet distance, which was originally
introduced by M. Fr$\acute{\mbox{e}}$chet in the study of shapes
of geometric spaces.
\end{rem}
Now we are in a position to state our main result.
\begin{thm}
\label{thm 2.1}Let $\mathcal{Z}$ be the space of P.L.T.s in $\mathbb{Z}^{N}$
equipped with the discrete $\sigma$-algebra. Then there exists a
sequence $\{\mathcal{T}^{(n)}\}$ of $\mathcal{Z}$-valued $\mathcal{G}_{1}$-measurable
random variables (random P.L.T.s), such that with probability one,
$\frac{1}{n}\cdot\mathcal{T}^{(n)}$ converges in trajectory to the
diffusion paths $(X_{t})_{0\leqslant t\leqslant1}.$
\end{thm}

It seems that the statement of Theorem \ref{thm 2.1} does not contain
much information about the approximating sequence $\{\mathcal{T}^{(n)}\}.$
However, when from the proof in the next section, we will see that
$\mathcal{T}^{(n)}$ is constructed in a quite explicit way.

It should be pointed out that the result of Theorem \ref{thm 2.1}
was already implicitly proved in \cite{lejan2012stratonovich} for
the case of Brownian motion.

A direct consequence of Theorem \ref{thm 2.1} is the result based
on Y. Le Jan and Z. Qian's formulation.
\begin{thm}
$\mathcal{F}_{1}=\mathcal{G}_{1}.$
\end{thm}

Before proving our main result Theorem \ref{thm 2.1} in the next
section, we first illustrate the idea and main steps of the proof.

We adopt the scheme and the key observation that the diffusion paths
can be recovered by reading out the maximal sequence of well-chosen
compactly supported differential $1$-forms such that the iterated
Stratonovich's integral of those $1$-forms (extended Stratonovich's
signature) along the diffusion paths over the duration of visiting
their supports is nonzero , which were proposed in \cite{lejan2012stratonovich}. 

The idea of the proof of Theorem \ref{thm 2.1} is the following. 

Firstly, decompose the Euclidean space $\mathbb{R}^{N}$ into disjoint
small boxes and narrow tunnels. By recording the successive visit
times of those small boxes, we can construct a piecewise linear approximation
of the diffusion paths. A convergence theorem can be proved by developing
certain types of estimates of rare events for the diffusion process.
By enlarging the size of those small boxes a little bit (by a higher
order infinitesimal relative to the size of boxes), we can similarly
get another piecewise linear approximation also converging to the
diffusion paths as the size of boxes goes to zero. Secondly, we construct
a family of ``special'' differential 1-forms on $\mathbb{R}^{N}$
(depending on the size of boxes) in a way that for any larger box,
we construct a $1$-form supported in it such that it is highly non-degenerate
on the inner smaller box. The crucial observation is that the Stratonovich's
integral of any of those $1$-forms along the diffusion paths over
the duration of visit of its support is nonzero. It turns out that
for a diffusion path, we can read out an associated unique maximal
finite sequence of $1$-forms (a P.L.T.) recording a sequence of boxes
in order such that the iterated Stratonovich's integral of this sequence
of $1$-forms (extended Stratonovich's signature) along the diffusion
path over the duration of visiting their supports is nonzero. It provides
us with sufficient information to recover the diffusion path by taking
limit in a reasonable way. This is due to the fact that based on our
construction, we can prove that such a maximal sequence always ``lies''
between the two piecewise linear approximations constructed before,
both of which converge to the diffusion path. Here we need to develop
a kind of squeeze theorem for the type of convergence (convergence
in trajectory in the setting of P.L.T.s defined as before) in our
situation.

To carry out the above idea, we are going to establish the following
three steps.

(1) Step one: proving a convergence result for the piecewise linear
approximation based on successive visit times of small boxes.

The proof consists of two ingredients. The first one is a probabilistic
estimate of the number of boxes visited over the time duration $[0,1],$
which can be developed by using a random time change technique. It
turns out that we can reduce to the Brownian motion case. The importance
of such an estimate is that we can get an asymptotic rate of the probability
that the number of boxes visited over $[0,1]$ is quite large. The
second one is the probabilistic estimate of the uniform distance between
the piecewise linear approximation path and the original diffusion
path, provided the number of boxes visited over $[0,1]$ is fixed.
This can be done by using the Strong Markov property and a quantitative
result in \cite{ben1984poisson} by G. Ben Arous, S. Kusuoka and D.W.
Stroock, which gives us control on the density of the Poisson kernel
of a given bounded domain in $\mathbb{R}^{N}$ and enables us to estimate
the probability that the diffusion process travels through narrow
tunnels. Combining the two ingredients, it is not hard to prove the
convergence result by using the Borel-Cantelli's lemma via a subsequence.

(2) Step two: constructing special differential $1$-forms and using
extended Stratonovich's signatures.

For any larger box, we are going to construct a suitable differential
$1$-form supported in it and highly nondegenerate on the inner smaller
box. The construction of such a differential $1$-form can be reduced
to the construction of a differential $1$-form such that the generator
of some associated SDE with dimension $N+1$ is hypoelliptic on the
support of the differential $1$-form. The family of diffenrential
$1$-forms constructed in such a way will be used to construct extended
Stratonovich's signatures, which in turn will be used to recover the
diffusion paths as stated in the idea of the proof.

(3) Step three: proving a squeeze theorem for convergence in trajectory
to recover the diffusion paths.

From the above two steps we constructed two sequences of piecewise
linear approximations of the diffusion paths, and between which a
sequence of P.L.T.s in terms of extended Stratonovich signatures.
We will formulate the term ``lying between'' in a rigorous way in
the setting of P.L.T.s and prove a squeeze theorem for convergence
in trajectory which fits our situation. Here the squeeze theorem we
are going to prove is not in the most general case (we need to make
use of special parametrizations), so we need to modify the piecewise
linear approximation associated to larger boxes to fit our case.

An advantage of using such a squeeze theorem is that we can get around
the estimates based on potential theory and partial differential equations,
which was used in \cite{lejan2012stratonovich} for the Laplace operator.
In fact, for a general elliptic operator $L,$ the associated partial
differential equation may not be well-posed and the conditions to
ensure a (regular) probabilistic representation of a solution is quite
restrictive and technical.

\section{Proof of the main result}

In this section, we will give the detailed proof of our main result
Theorem \ref{thm 2.1}.

Recall that $\{X_{t}:t\geqslant0\}$ is an $N$-dimensional diffusion
process defined by the following SDE:
\[
dX_{t}=\sum_{\alpha=1}^{d}V_{\alpha}(X_{t})\circ dW_{t}^{\alpha}+V_{0}(X_{t})dt
\]
with $X_{0}=0,$ in which the generating vector fields satisfy Assumptions
(A), (B), (C).

In the following the coordinates of $x\in\mathbb{R}^{N}$ is taken
with respect to the orthonormal basis given in Assumption (C).

\subsection{Discretization and an approximation result}

Similar to the idea of Y. Le Jan and Z. Qian, we first construct a
suitable approximation scheme for the diffusion paths. 

For convenience, a constant is called universal if it depends only
on the generator $L$ and the dimensions $N,d.$ Moreover, sometimes
we may use the same notation to denote universal constants coming
out from estimates, although they may be different from line to line.

Let $0<\varepsilon<1.$ For $z=(z_{1},\cdots,z_{N})\in\mathbb{Z}^{N},$
let $H_{z}^{\varepsilon}$ be the $N$-cube in $\mathbb{R}^{N}$ defined
by 
\[
H_{z}^{\varepsilon}=\{(x_{1},\cdots,x_{N}):\ \varepsilon z_{i}-\frac{\varepsilon-\varepsilon^{\mu}}{2}\leqslant x_{i}\leqslant\varepsilon z_{i}+\frac{\varepsilon-\varepsilon^{\mu}}{2},\ i=1,2,\cdots,N\},
\]
where $\mu$ is some universal constant to be chosen later on. 

For
technical reasons we assume that the boundary of $H_{z}^{\varepsilon}$ is smoothed to the order of $\varepsilon^{2\mu}$. Such a smoothing procedure can be done in a simple geometric way, or by using standard mollifiers. In the case of $N=2,$ we
can simply replace each corner of $H_{z}^{\varepsilon}$ by a quarter
of a circle with radius  $\varepsilon^{2\mu}$. The
space $\mathbb{R}^{N}$ is then divided into disjoint small boxes
and narrow tunnels.

Now we are going to construct an approximation of diffusion paths
$X_{t}$ over the time duration $[0,1].$

Let $\tau_{0}^{\varepsilon}=0$ and $\boldsymbol{m}_{0}^{\varepsilon}=(0,\cdots,0).$
For $k\geqslant1,$ define 
\[
\tau_{k}^{\varepsilon}=\inf\{t>\tau_{k-1}^{\varepsilon}:\ X_{t}\in\bigcup_{z\neq\boldsymbol{m}_{k-1}^{\varepsilon}}H_{z}^{\varepsilon}\}.
\]
If $\tau_{k}^{\varepsilon}<\infty,$ define $ $$\boldsymbol{m}_{k}^{\varepsilon}$
be the integer point in $\mathbb{Z}^{N}$ such that $X_{\tau_{k}^{\varepsilon}}\in H_{\boldsymbol{m}_{k}^{\varepsilon}}^{\varepsilon};$
if $\tau_{k}^{\varepsilon}=\infty,$ define $\boldsymbol{m}_{k}^{\varepsilon}=\boldsymbol{m}_{k-1}^{\varepsilon}.$
Intuitively, the sequence of hitting times $\{\tau_{k}^{\varepsilon}\}_{k=0}^{\infty}$
records the successive visit times of the small boxes and the sequence
of integer points $\{\boldsymbol{m}_{k}^{\varepsilon}\}_{k=0}^{\infty}$
records the boxes visited by the diffusion paths in order (revisit
of the same box before visiting other boxes doesn't count). Note that
it is possible that $P(\tau_{k}^{\varepsilon}=\infty)>0$ since with
positive probability the process can always stay in narrow tunnels
after leaving some box.

Let $M_{H}^{\varepsilon}$ be the number of boxes visited by the diffusion
paths over the time duration $[0,1].$ Formally,
\[
M_{H}^{\varepsilon}=\inf\{k\geqslant0:\ \tau_{k+1}^{\varepsilon}>1\}.
\]
It follows from uniform continuity of the diffusion paths over $[0,1]$
that $M_{H}^{\varepsilon}<\infty$ for almost surely.

By a standard random time change argument, we can prove the following.
\begin{lem}
\label{lem 3.1} Let $C=\max\{\|V_{1}\|_{\infty},\cdots,\|V_{d}\|_{\infty},\|V_{0}+\frac{1}{2}\sum_{\alpha=1}^{d}\nabla_{V_{\alpha}}V_{\alpha}\|_{\infty}\}.$
Then for any $k>\frac{2C}{\varepsilon^{\mu}},$
\[
P(M_{H}^{\varepsilon}=k)\leqslant4Nke^{-\frac{\varepsilon^{2\mu}k}{8NdC^{2}}}.
\]
\end{lem}
\begin{proof}
For $k\geqslant1,$ it is obvious that 
\begin{eqnarray*}
P(M_{H}^{\varepsilon}=k) & = & P(\tau_{k}^{\varepsilon}\leqslant1,\ \tau_{k+1}^{\varepsilon}>1)\\
 & \leqslant & P(\bigcup_{l=1}^{k}\{\tau_{l}^{\varepsilon}-\tau_{l-1}^{\varepsilon}\leqslant\frac{1}{k},\ \tau_{k}^{\varepsilon}\leqslant1\})\\
 & \leqslant & \sum_{l=1}^{k}P(\tau_{l}^{\varepsilon}-\tau_{l-1}^{\varepsilon}\leqslant\frac{1}{k},\ \tau_{k}^{\varepsilon}\leqslant1)\\
 & \leqslant & \sum_{l=1}^{k}P(\sup_{0\leqslant t\leqslant1/k}|X_{t+\tau_{l-1}^{\varepsilon}}-X_{\tau_{l-1}^{\varepsilon}}|\geqslant\varepsilon^{\mu},\ \tau_{l-1}^{\varepsilon}<\infty),
\end{eqnarray*}
where the last inequality comes from the fact that the distance between
two different boxes is bounded from below by $\varepsilon^{\mu}.$
By the strong Markov property, it suffices to estimate 
\[
P(\sup_{0\leqslant t\leqslant1/k}|X_{t}-x|\geqslant\varepsilon^{\mu}),
\]
 where $X_{t}$ is the diffusion process defined by (\ref{SDE}) starting
at $x\in\mathbb{R}^{N}.$ 

By rewriting (\ref{SDE}) in the sense of It$\hat{\mbox{o}}$, we
have
\[
\begin{cases}
dX_{t}=\sum_{\alpha=1}^{d}V_{\alpha}(X_{t})dW_{t}^{\alpha}+\widetilde{V_{0}}(X_{t})dt,\\
X_{0}=x,
\end{cases}
\]
where $\widetilde{V_{0}}=V_{0}+\frac{1}{2}\sum_{\alpha=1}^{d}\nabla_{V_{\alpha}}V_{\alpha}$.
It follows that for $k>\frac{2C}{\varepsilon^{\mu}}$, we have 
\begin{align*}
 & P(\sup_{0\leqslant t\leqslant1/k}|X_{t}-x|\geqslant\varepsilon^{\mu})\\
= & P(\sup_{0\leqslant t\leqslant1/k}|\sum_{\alpha=1}^{d}\int_{0}^{t}V_{\alpha}(X_{s})dW_{s}^{\alpha}+\int_{0}^{t}\widetilde{V_{0}}(X_{s})ds|\geqslant\varepsilon^{\mu})\\
\leqslant & P(\sup_{0\leqslant t\leqslant1/k}|\sum_{\alpha=1}^{d}\int_{0}^{t}V_{\alpha}(X_{s})dW_{s}^{\alpha}|\geqslant\frac{\varepsilon^{\mu}}{2})\\
\leqslant & \sum_{i=1}^{N}P(\sup_{0\leqslant t\leqslant1/k}|\sum_{\alpha=1}^{d}\int_{0}^{t}V_{\alpha}^{i}(X_{s})dW_{s}^{\alpha}|\geqslant\frac{\varepsilon^{\mu}}{2\sqrt{N}}).
\end{align*}
By using a standard random time change technique and the inequality
\[
\int_{x}^{\infty}\frac{1}{\sqrt{2\pi}}e^{-\frac{1}{2}t^{2}}dt\leqslant e^{-\frac{1}{2}x^{2}},\ \ \ x>0,
\]
it is then easy to obtain that 
\[
P(\sup_{0\leqslant t\leqslant1/k}|X_{t}-x|\geqslant\varepsilon^{\mu})\leqslant4Ne^{-\frac{\varepsilon^{2\mu}k}{8NdC^{2}}}.
\]

Therefore, we have 
\[
P(M_{H}^{\varepsilon}=k)\leqslant4Nke^{-\frac{\varepsilon^{2\mu}k}{8NdC^{2}}},\ \ \ k>\frac{2C}{\varepsilon^{\mu}},
\]
and the proof is complete.
\end{proof}

Now we define polygonal approximations of the diffusion paths through
successive visits of those boxes. More precisely, if $M_{H}^{\varepsilon}=0,$
define $X_{t}^{\varepsilon}\equiv(0,0,\cdots,0)$ on $[0,1];$ otherwise
for $1\leqslant k\leqslant M_{H}^{\varepsilon}$, define
\[
X_{t}^{\varepsilon}=\frac{\tau_{k}^{\varepsilon}-t}{\tau_{k}^{\varepsilon}-\tau_{k-1}^{\varepsilon}}\varepsilon\boldsymbol{m}_{k}^{\varepsilon}+\frac{t-\tau_{k-1}^{\varepsilon}}{\tau_{k}^{\varepsilon}-\tau_{k-1}^{\varepsilon}}\varepsilon\boldsymbol{m}_{k-1}^{\varepsilon},\ \ \ t\in[\tau_{k-1}^{\varepsilon},\tau_{k}^{\varepsilon}],
\]
and on $[\tau_{M_{H}^{\varepsilon}},1],$ define $X_{t}^{\varepsilon}\equiv\varepsilon\boldsymbol{m}_{M_{H}^{\varepsilon}}^{\varepsilon}.$
Figure 1 illustrates the construction.

\begin{figure}
\begin{center}

\includegraphics[scale=0.35]{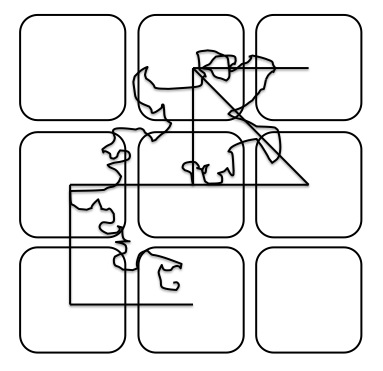}
\caption{This figure illustrates the construction of the polygonal approximation
of the diffusion path. Here the total number of boxes visited by the
path in order is $8$.}

\end{center}
\end{figure}

Now we have the following convergence result. The proof is developed
for arbitrary dimension $N\geqslant2,$ but in the case of $N=2$
the idea is easier to visualize.
\begin{prop}
\label{prop 3.1}There exists a subsequence $\varepsilon_{n}\rightarrow0,$
such that with probability one, $(X_{t}^{\varepsilon_{n}})_{0\leqslant t\leqslant1}$
converges uniformly to the diffusion paths $(X_{t})_{0\leqslant t\leqslant1}$
on $[0,1]$ as $n\rightarrow\infty.$\end{prop}
\begin{proof}
We aim at estimating the following probability 
\[
P(\sup_{0\leqslant t\leqslant1}|X_{t}^{\varepsilon}-X_{t}|>\lambda\varepsilon),
\]
where $\lambda$ is a large universal constant to be chosen later
on. For convenience, we will assume that $\frac{\lambda}{12}$ is
a positive integer. 

For this purpose, let $k$ be a large integer to be chosen later on
(may depend on $\varepsilon$). It follows that

\begin{align*}
 & P(\sup_{0\leqslant t\leqslant1}|X_{t}^{\varepsilon}-X_{t}|>\lambda\varepsilon)\\
\leqslant & \sum_{l=0}^{k}P(\sup_{0\leqslant t\leqslant1}|X_{t}^{\varepsilon}-X_{t}|>\lambda\varepsilon,M_{H}^{\varepsilon}=l)+P(M_{H}^{\varepsilon}>k)\\
\leqslant & \sum_{l=0}^{k}P(\bigcup_{j=1}^{l}\{\sup_{\tau_{j-1}^{\varepsilon}\leqslant t\leqslant\tau_{j}^{\varepsilon}}|X_{t}^{\varepsilon}-X_{t}|>\lambda\varepsilon,\ M_{H}^{\varepsilon}=l\}\\
 & \bigcup\{\sup_{\tau_{l}^{\varepsilon}\leqslant t\leqslant1}|X_{t}^{\varepsilon}-X_{t}|>\lambda\varepsilon,\ M_{H}^{\varepsilon}=l\})+P(M_{H}^{\varepsilon}>k)
\end{align*}
\begin{align*}
\leqslant & \sum_{l=0}^{k}[\sum_{j=1}^{l}P(\sup_{\tau_{j-1}^{\varepsilon}\leqslant t\leqslant\tau_{j}^{\varepsilon}}|X_{t}^{\varepsilon}-X_{t}|>\lambda\varepsilon,\ \tau_{j}^{\varepsilon}\leqslant1)\\
 & +P(\sup_{\tau_{l}^{\varepsilon}\leqslant t\leqslant1}|X_{t}^{\varepsilon}-X_{t}|>\lambda\varepsilon,\ M_{H}^{\varepsilon}=l)]+P(M_{H}^{\varepsilon}>k).
\end{align*}

We first estimate $P(\sup_{\tau_{j-1}^{\varepsilon}\leqslant t\leqslant\tau_{j}^{\varepsilon}}|X_{t}^{\varepsilon}-X_{t}|>\lambda\varepsilon,\ \tau_{j}^{\varepsilon}\leqslant1).$
The idea is the following: the event $\{\sup_{\tau_{j-1}^{\varepsilon}\leqslant t\leqslant\tau_{j}^{\varepsilon}}|X_{t}^{\varepsilon}-X_{t}|>\lambda\varepsilon,\ \tau_{j}^{\varepsilon}\leqslant1\}$
implies that after time $\tau_{j-1}^{\varepsilon},$ the process must
have travelled through many narrow tunnels and spread far away from
$H_{\boldsymbol{m}_{j-1}^{\varepsilon}}^{\varepsilon}$ by many boxes
before visiting another box. Define $\sigma_{0}$ to be the first
time after $\tau_{j-1}^{\varepsilon}$ that the process arrives at
the entrance of some narrow tunnel which is far away from $H_{\boldsymbol{m}_{j-1}^{\varepsilon}}^{\varepsilon}$
with distance at least $\frac{\lambda}{6}\varepsilon$ without hitting
any other boxes. For $ $$1\leqslant L\leqslant\lambda/12,$ define
$\sigma_{L}$ to be the first time after $\sigma_{L-1}$ that the
process travels through a narrow tunnel without hitting any boxes
other than $H_{\boldsymbol{m}_{j-1}^{\varepsilon}}^{\varepsilon}$
(define $\sigma_{0}=\infty$ if there is no such arrival and $\sigma_{L}=\infty$
if there is no such travel through). It is easy to see that 
\[
\{\sup_{\tau_{j-1}^{\varepsilon}\leqslant t\leqslant\tau_{j}^{\varepsilon}}|X_{t}^{\varepsilon}-X_{t}|>\lambda\varepsilon,\ \tau_{j}^{\varepsilon}\leqslant1\}\subset\{\sigma_{0}<\sigma_{1}<\cdots<\sigma_{\lambda/12}<\infty\}.
\]
Thus it suffices to estimate $P(\sigma_{0}<\sigma_{1}<\cdots<\sigma_{\lambda/12}<\infty).$
This can be done by using the strong Markov property and a quantitative
estimate for the Poisson kernel of some nice domain in \cite{ben1984poisson}.
In fact, by the strong Markov property, 
\begin{align*}
 & P(\sigma_{0}<\sigma_{1}<\cdots<\sigma_{\lambda/12}<\infty)\\
= & E[P(\sigma_{0}<\sigma_{1}<\cdots<\sigma_{\lambda/12}<\infty|\mathcal{F}_{\sigma_{\lambda/12-1}}^{X}),\ \sigma_{\lambda/12-1}<\infty]\\
= & E[P^{X_{\sigma_{\lambda/12-1}}(\omega)}(\{\omega':\ \omega'\in T(\omega)\}),\ \sigma_{0}(\omega)<\sigma_{1}(\omega)<\cdots<\sigma_{\lambda/12-1}(\omega)<\infty],
\end{align*}
where $T(\omega)$ denotes the set of sample paths $\omega'$ of the
diffusion process starting at $X_{\sigma_{\lambda/12-1}}(\omega)$
such that the first time of traveling through a narrow tunnel without
hitting any boxes is finite. By the assumptions on the generating
vector fields, the generator $L$ and those small boxes $H_{z}^{\varepsilon}$
verify the conditions of Lemma 2.6 in \cite{ben1984poisson}. It follows
from the lemma that the Poisson kernel $H(x,d\eta)$ of any small
box has a continuous density $h(x,\eta)$ with respect to the normalized
surface measure $d\eta.$ Moreover, there are universal constants
(in particular, not depending on $\varepsilon$) $K_{0},\nu_{0}>0,$
such that 
\[
|h(x,\eta)|\leqslant K_{0}\cdot\mbox{dist\ensuremath{(x,\partial G)/\mbox{dist\ensuremath{(x,\eta)^{\nu_{0}}}}}},
\]
for any $x$ in the box and $\eta$ on the boundary. Since traveling
through narrow tunnels implies escaping through narrow windows of
the boundary of some associated domain, it follows that on $\{\omega:\ \sigma_{\lambda/12-1}(\omega)<\infty\},$
\[
P^{X_{\sigma_{\lambda/12-1}}(\omega)}(\{\omega':\ \omega'\in T(\omega)\})\leqslant K(\varepsilon-\varepsilon^{\mu}-2\varepsilon^{2\mu})^{1-\nu_{0}}\cdot\frac{\varepsilon^{\mu}}{\varepsilon},
\]
for some universal constant $K>0.$ Now it is clear that if we choose
$\mu$ to be universal and far greater than $\nu_{0},$ then on $\{\omega:\ \sigma_{\lambda/12-1}(\omega)<\infty\},$
we have 
\[
P^{X_{\sigma_{\lambda/12-1}}(\omega)}(\{\omega':\ \omega'\in T(\omega)\})\leqslant K\varepsilon^{\mu-\nu_{0}},
\]
for some universal constant $K>0.$ Therefore,
\[
P(\sigma_{0}<\sigma_{1}<\cdots<\sigma_{\lambda/12}<\infty)\leqslant K\varepsilon^{\mu-\nu_{0}}P(\sigma_{0}<\sigma_{1}<\cdots<\sigma_{\lambda/12-1}<\infty).
\]
By induction, it is immediate that
\begin{eqnarray*}
P(\sigma_{0}<\sigma_{1}<\cdots<\sigma_{\lambda/12}<\infty) & \leqslant & K^{\frac{\lambda}{12}}\varepsilon^{\frac{\lambda}{12}(\mu-\nu_{0})}P(\sigma_{0}<\infty)\\
 & \leqslant & K^{\frac{\lambda}{12}}\varepsilon^{\frac{\lambda}{12}(\mu-\nu_{0})}.
\end{eqnarray*}
 Therefore, we arrive at 
\[
P(\sup_{\tau_{j-1}^{\varepsilon}\leqslant t\leqslant\tau_{j}^{\varepsilon}}|X_{t}^{\varepsilon}-X_{t}|>\lambda\varepsilon,\ \tau_{j}^{\varepsilon}\leqslant1)\leqslant K^{\frac{\lambda}{12}}\varepsilon^{\frac{\lambda}{12}(\mu-\nu_{0})}.
\]

The estimate of $P(\sup_{\tau_{l}^{\varepsilon}\leqslant t\leqslant1}|X_{t}^{\varepsilon}-X_{t}|>\lambda\varepsilon,\ M_{H}^{\varepsilon}=l)$
is exactly the same as above since on $\{M_{H}^{\varepsilon}=l\},$
there will be no visit of boxes other than $H_{\boldsymbol{m}_{l}^{\varepsilon}}^{\varepsilon}$
during $[\tau_{l}^{\varepsilon},1].$

Now consider $P(M_{H}^{\varepsilon}>k).$ By Lemma \ref{lem 3.1},
if $k>\frac{2C}{\varepsilon^{\mu}}$,
\begin{eqnarray*}
P(M_{H}^{\varepsilon}>k) & = & \sum_{l=k+1}^{\infty}P(M_{H}^{\varepsilon}=l)\\
 & \leqslant & \sum_{l=k+1}^{\infty}4Nle^{-\frac{\varepsilon^{2\mu}l}{8NdC^{2}}}\\
 & \leqslant & \frac{4Ne^{-k\tilde{C}\varepsilon^{2\mu}}}{(1-e^{-\tilde{C}\varepsilon^{2\mu}})^{2}}+\frac{4Nke^{-k\tilde{C}\varepsilon^{2\mu}}}{1-e^{-\tilde{C}\varepsilon^{2\mu}}},
\end{eqnarray*}
where $\tilde{C}=\frac{1}{8NdC^{2}}.$ Choose a universal constant
$\gamma>>2\mu,$ and let $k=[\frac{1}{\varepsilon^{\gamma}}]$ (when
$\varepsilon$ is small, the condition $k>\frac{2C}{\varepsilon^{\mu}}$
in Lemma \ref{lem 3.1} is satisfied). It follows that 
\[
P(M_{H}^{\varepsilon}>k)\leqslant C'(\frac{e^{-\frac{\tilde{C}}{\varepsilon^{\gamma-2\mu}}}}{(1-e^{-\tilde{C}\varepsilon^{2\mu}})^{2}}+\frac{1}{\varepsilon^{\gamma}}\frac{e^{-\frac{\tilde{C}}{\varepsilon^{\gamma-2\mu}}}}{1-e^{-\tilde{C}\varepsilon^{2\mu}}}),
\]
where $C'$ is a positive universal constant.

Combining with the estimates before, we arrive at 
\begin{align}
 & P(\sup_{0\leqslant t\leqslant1}|X_{t}^{\varepsilon}-X_{t}|>\lambda\varepsilon)\nonumber \\
\leqslant & C'(K^{\frac{\lambda}{12}}\varepsilon^{\frac{\lambda}{12}(\mu-\nu_{0})-2\gamma}+\frac{e^{-\frac{\tilde{C}}{\varepsilon^{\gamma-2\mu}}}}{(1-e^{-\tilde{C}\varepsilon^{2\mu}})^{2}}+\frac{1}{\varepsilon^{\gamma}}\frac{e^{-\frac{\tilde{C}}{\varepsilon^{\gamma-2\mu}}}}{1-e^{-\tilde{C}\varepsilon^{2\mu}}}).\label{convergence}
\end{align}
Finally, choose a positive universal integer $\lambda$ such that
\[
\lambda>\frac{24\gamma+24}{\mu-\nu_{0}}
\]
and $\frac{\lambda}{12}$ is a positive integer. By taking $\varepsilon_{n}=1/n,$
we have 
\[
\sum_{n=1}^{\infty}P(\sup_{0\leqslant t\leqslant1}|X_{t}^{\varepsilon_{n}}-X_{t}|>\lambda\varepsilon_{n})<\infty.
\]
Borel-Cantelli's lemma then yields the desired result.

Now the proof is complete.
\end{proof}

Let $\mu'>\mu$ be another universal constant. Define $(V_{z}^{\varepsilon},\zeta_{k}^{\varepsilon},\boldsymbol{n}_{k}^{\varepsilon},M_{V}^{\varepsilon},\widetilde{X}^{\varepsilon})$
in the same way as $(H_{z}^{\varepsilon},\tau_{k}^{\varepsilon},\boldsymbol{m}_{k}^{\varepsilon},M_{H}^{\varepsilon},X^{\varepsilon})$
only with $\mu$ replaced by $\mu'$, then Proposition \ref{prop 3.1}
is also true for $\widetilde{X}^{\varepsilon}$ (with $\varepsilon_{n}=\frac{1}{n}$
as in the proof of Proposition \ref{prop 3.1}).

To complete the proof, it remains to trace the diffusion paths via
extended Stratonovich's signatures ``between'' smaller boxes $H_{z}^{\varepsilon}$
and larger boxes $V_{z}^{\varepsilon},$ and prove a squeeze theorem
so that we are able to pass to the same limit $X_{t}.$

\subsection{Constructing differential 1-forms and using extended Stratonovich's
signatures}

To trace the diffusion paths by using extended Stratonovich's signatures,
we first need to construct suitable compactly supported differential
$1$-forms, such that the Stratonovich's integral of any such $1$-form
$\phi$ along the diffusion paths over the duration of visit of $\mbox{supp \ensuremath{\phi}}$
is nonzero with probability one.

To this end, it suffices to construct a suitable differential 1-form
$\phi$ on $\mathbb{R}^{N}$ with compact support such that the family
of vector fields on $\mathbb{R}^{N+1}$:
\[
\{\left(\begin{array}{c}
V_{1}\\
\phi\cdot V_{1}
\end{array}\right),\cdots,\left(\begin{array}{c}
V_{d}\\
\phi\cdot V_{d}
\end{array}\right);\left(\begin{array}{c}
V_{0}\\
\phi\cdot V_{0}
\end{array}\right)\}
\]
satisfies H$\ddot{\mbox{o}}$rmander's condition on $(\mbox{supp \ensuremath{\phi)}}\times\mathbb{R}^{1}$
so that the generator of the diffusion process on $\mathbb{R}^{N+1}$
defined by 
\[
\begin{cases}
dX_{t}=\sum_{\alpha=1}^{d}V_{\alpha}(X_{t})\circ dW_{t}^{\alpha}+V_{0}(X_{t})dt,\\
dX_{t}^{N+1}=\phi(X_{t})\circ dX_{t},
\end{cases}
\]
is hypoelliptic on $(\mbox{supp \ensuremath{\phi)}}\times\mathbb{R}^{1}$,
which ensures the existence of smooth probability densities of certain
Wiener functionals. Here and thereafter we use the geometric notation
for convenience (so $V_{\alpha}$ is a regarded as a column vector
and $\phi$ is regarded as a row vector in $\mathbb{R}^{N}$). In
fact, if this is possible, then we can proceed in the same way as
Lemma 2.1, Lemma 2.2 and Lemma 2.3 in \cite{lejan2012stratonovich}
to show that the Stratonovich's integral of $\phi$ along the diffusion
paths over the duration of visit of $\mbox{supp \ensuremath{\phi}}$
is nonzero with probability one, since starting from this point the
proof relies only on the strong Markov property and again the results
in \cite{ben1984poisson}, which hold true from our assumptions on
the generating vector fields $\{V_{1},\cdots,V_{d};V_{0}\}.$ 

Now to make it more precise, for $z\in\mathbb{Z}^{N}$ and $\varepsilon>0,$
we are interested in constructing a differential $1$-form $\phi_{z}^{\varepsilon}$
such that 
\[
H_{z}^{\varepsilon}\subset(\mbox{supp \ensuremath{\phi_{z}^{\varepsilon}}})^{\circ}\subset\mbox{supp \ensuremath{\phi_{z}^{\varepsilon}}}\subset(V_{z}^{\varepsilon})^{\circ},
\]
and $\phi_{z}^{\varepsilon}$ has the property mentioned before.

The following result gives the desired construction.

\begin{prop}
\label{prop 3.2}Assume that the family of vector field $\{V_{1},\cdots,V_{d};V_{0}\}$
satisfies H$\ddot{\mbox{o}}$rmander's condition at every point $x$
in $\mathbb{R}^{N}$. Let $G$ be a bounded domain in $\mathbb{R}^{N}$
and $W$ be an open subdomain of $G$ such that 
\[
W\subset\subset G.
\]
Let $\eta\in C_{0}^{\infty}(\mathbb{R}^{N})$ be a cut-off function
of $\overline{W}$, that is, $0\leqslant\eta\leqslant1,$ $\eta\equiv1$
on $\overline{W}$ and $\eta=0$ outside a small neighborhood of $\overline{W}$.
Then there exists $\Lambda>0,$ such that for any $\xi\in\mathbb{R}^{N}$
with $|\xi|>\Lambda$, if we define the differential $1$-form $\phi$
on $\mathbb{R}^{N}$ by 
\begin{equation}
\phi(x)=\eta(x)e^{-\frac{1}{2}|x-\xi|^{2}}(dx^{1}+\cdots+dx^{N}),\label{1-form}
\end{equation}
and define the vector field $\widetilde{V_{\alpha}}$ on $\mathbb{R}^{N+1}$
(independent of $x^{N+1}$) by
\begin{equation}
\widetilde{V}_{\alpha}=\left(\begin{array}{c}
V_{\alpha}\\
\phi\cdot V_{\alpha}
\end{array}\right),\ \ \ \alpha=0,1,\cdots d,\label{vector fields in R^(N+1)}
\end{equation}
then the family of vector fields 
\[
\{\widetilde{V}_{1},\cdots,\widetilde{V}_{d};\widetilde{V}_{0}\}
\]
satisfies H$\ddot{\mbox{o}}$rmander's condition at every point on
$(\mbox{supp \ensuremath{\phi)\times\mathbb{R}}}$. In other words,
the differential operator $\widetilde{L}$ on $\mathbb{R}^{N+1}$
defined by 
\begin{equation}
\widetilde{L}=\frac{1}{2}\sum_{\alpha=1}^{d}\widetilde{V}_{\alpha}^{2}+\widetilde{V}_{0}\label{tilde L}
\end{equation}
is hypoelliptic on $(\mbox{supp \ensuremath{\phi})\ensuremath{\times\mathbb{R}}}$.\end{prop}
\begin{proof}
For a differential $1$-form $\phi$ on $\mathbb{R}^{N}$ defined
by (\ref{1-form}), define the vector fields
\[
\{\widetilde{V}_{1},\cdots,\widetilde{V}_{d};\widetilde{V}_{0}\}
\]
on $\mathbb{R}^{N+1}$ by (\ref{vector fields in R^(N+1)}). Note
that $\mbox{supp \ensuremath{\phi}}$ is independent of $\xi$, which
will be denoted by $K.$

Let
\begin{eqnarray*}
\Theta_{1} & = & \{1,2,\cdots,d\};\\
\Theta_{n} & = & \{(\alpha_{1},\cdots,\alpha_{n}):\ \alpha_{i}=0,1,\cdots,d\},\ n\geqslant2;\\
\Theta & = & \bigcup_{n=1}^{\infty}\Theta_{n}.
\end{eqnarray*}
For $\theta=(\theta_{1},\cdots,\theta_{n})\in\Theta_{n},$ denote
$|\theta|=n$, and we use the notation $V_{[\theta]}$ ($\widetilde{V}_{[\theta]}$,
respectively) to denote $[V_{\theta_{1}},[V_{\theta_{2}},\cdots,[V_{\theta_{n-1}},V_{\theta_{n}}]]]$
($[\widetilde{V}_{\theta_{1}},[\widetilde{V}_{\theta_{2}},\cdots,[\widetilde{V}_{\theta_{n-1}},\widetilde{V}_{\theta_{n}}]]]$,
respectively). 

We first prove that for any $\theta\in\Theta,$ $\widetilde{V}_{[\theta]}$
can be written as 
\[
\widetilde{V}_{[\theta]}=\left(\begin{array}{c}
V_{[\theta]}\\
g_{[\theta]}+\phi\cdot V_{[\theta]}
\end{array}\right)
\]
for some $g_{[\theta]}\in C_{b}^{\infty}(\mathbb{R}^{N+1})$ independent
of $x^{N+1}.$ In fact, when $\theta\in\Theta_{1},$ it is just the
definition of $\widetilde{V}_{[\theta]}$. Assume that it is true
for any $\theta\in\Theta_{n}.$ Let $\theta\in\Theta_{n+1},$ then
there exists some $0\leqslant\alpha\leqslant d$ and $\theta'\in\Theta_{n}$,
such that 
\[
V_{[\theta]}=[V_{\alpha},V_{[\theta']}],\ \widetilde{V}_{[\theta]}=[\widetilde{V}_{\alpha},\widetilde{V}_{[\theta']}].
\]
By the induction hypothesis, we have 
\begin{eqnarray*}
\widetilde{V}_{[\theta]} & = & \left[\left(\begin{array}{c}
V_{\alpha}\\
\phi\cdot V_{\alpha}
\end{array}\right),\left(\begin{array}{c}
V_{[\theta']}\\
g_{[\theta']}+\phi\cdot V_{[\theta']}
\end{array}\right)\right]\\
 & = & \left(\begin{array}{c}
[V_{\alpha},V_{[\theta']}]\\
\nabla^{N}(g_{[\theta']}+\phi\cdot V_{[\theta']})\cdot V_{\alpha}-\nabla^{N}(\phi\cdot V_{\alpha})\cdot V_{[\theta']}
\end{array}\right)\\
 & = & \left(\begin{array}{c}
V_{[\theta]}\\
g_{[\theta]}+\phi\cdot V_{[\theta]}
\end{array}\right),
\end{eqnarray*}
where 
\[
g_{[\theta]}=\nabla^{N}g_{[\theta']}\cdot V_{\alpha}+V_{[\theta']}^{T}\cdot\nabla^{N}\phi^{T}\cdot V_{\alpha}-V_{\alpha}^{T}\cdot\nabla^{N}\phi^{T}\cdot V_{[\theta']}\in C_{b}^{\infty}(\mathbb{R}^{N+1}),
\]
which is independent of $x^{N+1}$. Here $\nabla^{N}$ denotes the
gradient operator with respect to $x^{(N)}=(x^{1},\cdots,x^{N})$
and $(\cdot)^{T}$ denotes the transpose operator.

Now we are going to prove the result by a compactness argument. 

A key observation is that for any fixed $ $$\theta\in\Theta,$ let
$g_{[\theta]}\in C_{b}^{\infty}(\mathbb{R}^{N+1})$ be such that 
\[
\widetilde{V}_{[\theta]}=\left(\begin{array}{c}
V_{[\theta]}\\
g_{[\theta]}+\phi\cdot V_{[\theta]}
\end{array}\right)
\]
as in the previous discussion, then $g_{[\theta]}$ is of the form
\[
g_{[\theta]}(x)=p_{[\theta]}(\xi;x)e^{-\frac{1}{2}|x-\xi|^{2}},
\]
where $p_{[\theta]}(\xi;x)$ is a polynomial in $\xi=(\xi^{1},\cdots,\xi^{N})$
with $C_{b}^{\infty}$ coefficients depending only on $x^{(N)}.$
Here the degree of $p_{[\theta]}$ is at most $|\theta|-1.$ In other
words, 
\[
p_{[\theta]}(\xi;x)=\sum_{j=0}^{|\theta|-1}\sum_{|\boldsymbol{\alpha}|=j}c_{\boldsymbol{\alpha}}(x^{(N)})\xi^{\boldsymbol{\alpha}}.
\]

Fix $x_{0}=(x_{0}^{(N)},x_{0}^{N+1})\in K^{\circ}\times\mathbb{R}^{1},$
where $x_{0}^{(N)}=(x_{0}^{1},\cdots,x_{0}^{N})\in\mathbb{R}^{N}.$
By the hypoellipticity of $L$ and continuity, there exists a neighborhood
$U\subset K^{\circ}$ of $x_{0}^{(N)}$ and $\theta^{(1)},\cdots,\theta^{(N)}\in\Theta,$
such that for any $x^{(N)}\in U$, 
\[
\{V_{[\theta^{(1)}]}(x^{(N)}),\cdots,V_{[\theta^{(N)}]}(x^{(N)})\}
\]
constitutes a basis of $\mathbb{R}^{N}$. It follows that for any
$x\in U\times\mathbb{R}^{1},$ the family of vectors in $\mathbb{R}^{N+1}$
\[
\{\widetilde{V}_{[\theta^{(1)}]}(x),\cdots,\widetilde{V}_{[\theta^{(N)}]}(x)\}
\]
generate an $N$-dimensional subspace of $\mathbb{R}^{N+1}.$ Let
$M=\max\{|\theta^{(1)}|,\cdots,|\theta^{(N)}|\}.$ Again by the assumptions
on $L$ and continuity, it is possible to choose $\theta\in\Theta$
with $|\theta|>M$, such that 
\begin{equation}
\mbox{degree \ensuremath{(p_{[\theta]})}}>M\label{degree}
\end{equation}
in some compact neighborhood $\overline{U_{0}}\subset U$ of $x_{0}^{(N)}.$
In particular, the choice of $\theta$ and $U_{0}$ is independent
of the $\xi$ since the coefficents of $p_{[\theta]}$ are functions
of $x$ only. 

Now we are going to show that there exists $\Lambda>0,$ such that
when $\xi\in\mathbb{R}^{N}$ with $|\xi|>\Lambda$, the vector field
$\widetilde{V}_{[\theta]}$ cannot be generated by $ $$\{\widetilde{V}_{[\theta^{(1)}]},\cdots,\widetilde{V}_{[\theta^{(N)}]}\}$
in $U_{0}\times\mathbb{R}^{1},$ so that 
\[
\mbox{dim Span\{\ensuremath{\widetilde{V}_{[\theta^{(1)}]},\cdots,\widetilde{V}_{[\theta^{(N)}]},\widetilde{V}_{[\theta]}}\}}=N+1,
\]
which yields the hypoellipticity of $\widetilde{L}$ defined by \ref{tilde L}
in $U_{0}\times\mathbb{R}^{1}.$ 

To prove this, first notice that there exists $\lambda^{i}(x^{(N)})\in C_{b}^{\infty}(U_{0}),$
such that 
\[
V_{[\theta]}(x^{(N)})=\sum_{i=1}^{N}\lambda^{i}(x^{(N)})V_{[\theta^{(i)}]}(x^{(N)}),\ \ \ \mbox{for }x^{(N)}\in U_{0}.
\]
Moreover, from (\ref{degree}) it is easy to see that there exists
$\Lambda>0,$ such that 
\begin{equation}
p_{[\theta]}(\xi;x)\neq\sum_{i=1}^{N}\lambda^{i}(x^{(N)})p_{[\theta^{(i)}]}(\xi;x)\label{degree not equal}
\end{equation}
for $\xi\in\mathbb{R}^{N}$ with $|\xi|>\Lambda$ and $x\in U_{0}\times\mathbb{R}^{1}.$
If 
\[
\widetilde{V}_{[\theta]}(x_{1})\in\mbox{Span}\{\widetilde{V}_{[\theta^{(1)}]}(x_{1}),\cdots,\widetilde{V}_{[\theta^{(N)}]}(x_{1})\}
\]
for some $x_{1}\in U_{0}\times\mathbb{R}^{1},$ then we must have
\[
\widetilde{V}_{[\theta]}(x_{1})=\sum_{i=1}^{N}\lambda^{i}(x_{1}^{(N)})\widetilde{V}_{[\theta^{(i)}]}(x_{1}).
\]
It follows from simple calculation that 
\begin{equation}
g_{[\theta]}(x_{1})=\sum_{i=1}^{N}\lambda^{i}(x_{1}^{(N)})g_{[\theta^{(i)}]}(x_{1}).\label{g=00003Dlambda g}
\end{equation}
This is a contradiction to (\ref{degree not equal}) when $|\xi|>\Lambda.$
Therefore, $\widetilde{V}_{[\theta]}$ cannot be generated by $\{\widetilde{V}_{[\theta^{(1)}]},\cdots,\widetilde{V}_{[\theta^{(N)}]}\}$
in $U_{0}\times\mathbb{R}^{1}$ if we choose $\xi$ with $|\xi|>\Lambda$
in the definition of $\phi.$ 

The case when $x_{0}\in\partial K\times\mathbb{R}^{1}$ can be proved
in the same way by replacing $U_{0}$ with $U_{0}\cap K.$

Finally, combining with the above local results and by the compactness
of $K,$ we are able to choose $\Lambda>0$ (depending on $K$), such
that for any $\xi\in\mathbb{R}^{N}$ with $|\xi|>\Lambda,$ the differential
operator $\widetilde{L}$ is hypoelliptic on $K\times\mathbb{R}^{1}.$

Now the proof is complete.
\end{proof}

For $z\in\mathbb{Z}^{N}$ and $\varepsilon>0,$ by taking $W=H_{z}^{\varepsilon}$
and $G=V_{z}^{\varepsilon},$ we can construct a differential $1$-form
$\phi_{z}^{\varepsilon}$ supported in $G$ according to Lemma \ref{prop 3.2}
(just take some fixed admissible $\xi\in\mathbb{R}^{N}$ as in the
lemma). By proceeding in the same way as in \cite{lejan2012stratonovich},
we conclude that the Stratonovich's integral of $\phi_{z}^{\varepsilon}$
along the diffusion paths over the duration of visit of $\mbox{supp \ensuremath{\phi_{z}^{\varepsilon}}}$
is nonzero with probability one.

Now we are going to construct extended Stratonovich's signatures to
trace the original diffusion paths by using these differential $1$-forms
$\phi_{z}^{\varepsilon}$.

We first define extended Stratonovich's signatures formally. 

For smooth differential forms $\psi^{1},\cdots,\psi^{k}$ on $\mathbb{R}^{N}$,
the iterated Stratonovich's integral $[\psi^{1},\cdots,\psi^{k}]_{s,t}$
($0\leqslant s<t\leqslant1$) defined inductively by 
\[
[\psi^{1},\cdots,\psi^{k}]_{s,t}=\int_{s<u<t}[\psi^{1},\cdots,\psi^{k-1}]_{s,u}\psi^{k}(\circ dX_{u}),
\]
where 
\[
[\psi^{1}]_{s,t}=\sum_{i=1}^{N}\int_{s<u<t}\psi_{i}^{1}(X_{u})\circ dX_{u}^{i},
\]
is called an \textit{extended Stratonovich's signature} of the diffusion
process $X_{t}.$

The following lemma allows us to use extended Stratonovich's signatures
for our study. The case of Brownian motion was proved in \cite{lejan2012stratonovich},
but we can easily adopt the proof to the our case without changing
anything (in fact, the proof does not rely on probabilistic features,
but only on paths). Recall that $\mathcal{G}_{1}$ is the completion
of the $\sigma$-algebra generated by the Stratonovich's signature
of $X_{t}$ over $[0,1].$ 
\begin{lem}
If $\psi^{1},\cdots,\psi^{k}$ are smooth differential $1$-forms
on $\mathbb{R}^{N}$ with compact supports, then 
\[
[\psi^{1},\cdots,\psi^{k}]_{0,1}
\]
 is $\mathcal{G}_{1}$-measurable.\end{lem}
\begin{proof}
See \cite{lejan2012stratonovich}, Lemma 1.3.
\end{proof}

For $m\geqslant0,$ let 
\[
\mathcal{W}_{m}=\{(z_{0}=(0,\cdots,0),z_{1},\cdots,z_{m}):\ z_{i}\in\mathbb{Z}^{N},\ z_{i}\neq z_{i-1},\ i=1,2,\cdots,m\}.
\]
An element $(z_{0},z_{1},\cdots,z_{m})\in\mathcal{W}_{m}$ is called
an admissible word of length $m+1$. For $\varepsilon>0,$ define
the $\mathcal{G}_{1}$-measurable random variable $M^{\varepsilon}$
to be the supremum of those $m\geqslant0$ such that
\[
[\phi_{z_{0}}^{\varepsilon},\phi_{z_{1}}^{\varepsilon},\cdots,\phi_{z_{m}}^{\varepsilon}]_{0,1}\neq0
\]
for some admissible word $(z_{0},z_{1},\cdots,z_{m})\in\mathcal{W}_{m}.$
It follows that $M_{H}^{\varepsilon}\leqslant M^{\varepsilon}\leqslant M_{V}^{\varepsilon}$
for almost surely. For $m\geqslant0,$ $(z_{0},z_{1},\cdots,z_{m})\in\mathcal{W}_{m},$
define 
\[
A_{m;(z_{0},z_{1},\cdots,z_{m})}^{\varepsilon}=\{\omega:\ M^{\varepsilon}=m,\ [\phi_{z_{0}}^{\varepsilon},\phi_{z_{1}}^{\varepsilon},\cdots,\phi_{z_{m}}^{\varepsilon}]_{0,1}\neq0\},
\]
then $ $$\{A_{m;(z_{0},z_{1},\cdots,z_{m})}^{\varepsilon}:\ m\geqslant0,\ (z_{0},z_{1},\cdots,z_{m})\in\mathcal{W}_{m}\}$
are mutually disjoint $\mathcal{G}_{1}$-measurable sets whose union
is the whole space $\Omega.$ See \cite{lejan2012stratonovich} for
a more detailed discussion.

Let 
\[
\mathcal{W}=\bigcup_{m=0}^{\infty}\mathcal{W}_{m}
\]
be the space of admissible words. For $\varepsilon>0,$ define the
mapping $\mathcal{Y}^{\varepsilon}:\ \Omega\rightarrow\mathcal{W},$
\[
\mathcal{Y}^{\varepsilon}(\omega)=(z_{0},z_{1},\cdots,z_{m}),
\]
if $(z_{0},z_{1},\cdots,z_{m})$ is such that 
\[
\omega\in A_{m;(z_{0},z_{1},\cdots,z_{m})}^{\varepsilon}.
\]
It follows that $\mathcal{Y}^{\varepsilon}$ is well-defined and $\mathcal{G}_{1}$-measurable.
Intuitively, $\mathcal{Y}^{\varepsilon}$ is the maximal admissible
word such that the associated extended Stratonovich's signature is
nonzero. It is natural to believe that $\mathcal{Y}^{\varepsilon}$
records a reasonable amount of information of the diffusion paths
and as $\varepsilon\rightarrow0,$ it might be possible to recover
the diffusion paths.

\subsection{Completing the proof: a squeeze theorem for convergence in trajectory}

In Section 2, we defined piecewise linear trajectories (P.L.T.s),
parametrization of a P.L.T., and introduced the concept of convergence
in trajectory. In this section, we are going to show that if $\mathcal{Y}^{\varepsilon}$
is regarded as a P.L.T. in $\mathbb{Z}^{N}$, then by taking $\varepsilon_{n}=\frac{1}{n},$
with probability one, $\varepsilon_{n}\cdot\mathcal{Y}^{\varepsilon_{n}}$
converges in trajectory to $(X_{t})_{0\leqslant t\leqslant1},$ which
completes the proof of our main theorem.

Recall that a P.L.T. $\mathcal{T}$ is essentially a finite sequence
of points in $\mathbb{R}^{N}$ (not necessarily all distinct). 
\begin{defn}
For a P.L.T. $\mathcal{T},$ $\mathcal{T}^{-}$ is denoted as the
new P.L.T. by removing the last point of $\mathcal{T}.$ Let $\mathcal{T}_{1},\mathcal{T}_{2}$
be two P.L.T.s. $\mathcal{T}_{1}$ is called a \textit{sub}-P.L.T. of
$\mathcal{T}_{2}$ (denoted by $\mathcal{T}_{1}\prec\mathcal{T}_{2}$)
if $\mathcal{T}_{1}$ is a subsequence of $\mathcal{T}_{2}.$
\end{defn}

By the convergence result and the construction of $\phi_{z}^{\varepsilon}$
in the last two subsections, if we denote $\mathcal{X}^{\varepsilon}$
(respectively, $\widetilde{\mathcal{X}}^{\varepsilon}$) as the associated
P.L.T. of the piecewise linear path $X^{\varepsilon}$ (respectively,
$\widetilde{X}^{\varepsilon}$), then it is immediate that
\[
(\mathcal{X}^{\varepsilon})^{-}\prec\varepsilon\cdot\mathcal{Y}^{\varepsilon}\prec(\widetilde{\mathcal{X}}^{\varepsilon})^{-},
\]
with probability one, and $\mathcal{X}^{\varepsilon_{n}}$ and $\widetilde{\mathcal{X}}^{\varepsilon_{n}}$
both converges in trajectory to $(X_{t})_{0\leqslant t\leqslant1}$.
Therefore, it is natural to claim a certain kind of squeeze theorem
for convergence in trajectory so we may conclude that $\mathcal{Y}^{\varepsilon_{n}}$
also converges in trajectory to $(X_{t})_{0\leqslant t\leqslant1}$
with probability one. 

The following result is a squeeze theorem for convergence in trajectory
we are looking for, which is sufficient for our use.
\begin{prop}
\label{prop 3.3}Assume that $\{\mathcal{T}_{1}^{(n)}\}$, $\{\mathcal{T}_{2}^{(n)}\}$
are two sequence of P.L.T.s such that: 

(1) the first points of $\mathcal{T}_{1}^{(n)}$ and $\mathcal{T}_{2}^{(n)}$
are identical;

(2) the last two points of $\mathcal{T}_{i}^{(n)}$ are identical
($i=1,2$).

Let $\sigma_{i}^{(n)}$ be a parametrization of $\mathcal{T}_{i}^{(n)}$
($i=1,2$) such that the partition points in $\sigma_{1}^{(n)}$ belong
to the partition points in $\sigma_{2}^{(n)}$ and for any $t<1$
in $\sigma_{1}^{(n)},$
\[
\mathcal{T}_{1}^{(n)}(t|\sigma_{1}^{(n)})=\mathcal{T}_{2}^{(n)}(t|\sigma_{2}^{(n)}).
\]
(This assumption implies that $(\mathcal{T}_{1}^{(n)})^{-}\prec(\mathcal{T}_{2}^{(n)})^{-}$.)
Let $\{\mathcal{T}^{(n)}\}$ be a sequence of P.L.T.s such that 
\[
(\mathcal{T}_{1}^{(n)})^{-}\prec\mathcal{T}^{(n)}\prec(\mathcal{T}_{2}^{(n)})^{-},
\]
and $(\gamma_{t})_{0\leqslant t\leqslant1}$ be a continuous path
in $\mathbb{R}^{N}$ such that 
\[
\lim_{n\rightarrow\infty}\sup_{0\leqslant t\leqslant1}|\mathcal{T}_{i}^{(n)}(t|\sigma_{i}^{(n)})-\gamma_{t}|=0,\ \ \ i=1,2.
\]
Then we can choose a parametrization $\sigma^{(n)}$ of $\mathcal{T}^{(n)},$
such that 
\[
\lim_{n\rightarrow\infty}\sup_{0\leqslant t\leqslant1}|\mathcal{T}^{(n)}(t|\sigma^{(n)})-\gamma_{t}|=0.
\]
In particular, $\mathcal{T}^{(n)}$ converges in trajectory to $(\gamma_{t})_{0\leqslant t\leqslant1}.$\end{prop}
\begin{proof}
For any $\varepsilon>0,$ there exists $n_{0}>0,$ such that for any
$n>n_{0},$
\begin{equation}
\sup_{0\leqslant t\leqslant1}|\mathcal{T}_{i}^{(n)}(t|\sigma_{i}^{(n)})-\gamma_{t}|<\varepsilon,\ \ \ i=1,2.\label{gamma}
\end{equation}
On the other hand, it is obvious that we are able to construct a parametrization
$\sigma^{(n)}$ of $\mathcal{T}^{(n)},$ such that: 

(1) the partition points of $\sigma_{1}^{(n)}$ belong to the partition
points in $\sigma^{(n)},$ and for any $t<1$ in $\sigma_{1}^{(n)}$,
\[
\mathcal{T}_{1}^{(n)}(t|\sigma_{1}^{(n)})=\mathcal{T}^{(n)}(t|\sigma^{(n)});
\]

(2) the partition points in $\sigma^{(n)}$ belong to the partition
points in $\sigma_{2}^{(n)},$ and for any $t<1$ in $\sigma^{(n)},$
\[
\mathcal{T}^{(n)}(t|\sigma^{(n)})=\mathcal{T}_{2}^{(n)}(t|\sigma_{2}^{(n)}).
\]

Let $t_{n}$ be the largest time spot in $\sigma_{1}^{(n)}$ such
that $t_{n}<1$. Let $u_{n}<v_{n}$ be any consecutive time spots
in $\sigma^{(n)},$ then on $[u_{n},v_{n}]$ both $\mathcal{T}_{1}^{(n)}(\cdot|\sigma_{1}^{(n)})$
and $\mathcal{T}^{(n)}(\cdot|\sigma^{(n)})$ are linear. Therefore,
by an elementary result on the comparison for linear paths, we have
\begin{align}
 & \sup_{u_{n}\leqslant t\leqslant v_{n}}|\mathcal{T}_{1}^{(n)}(t|\sigma_{1}^{(n)})-\mathcal{T}^{(n)}(t|\sigma^{(n)})|\label{comparison}\\
\leqslant & \max\{|\mathcal{T}_{1}^{(n)}(u_{n}|\sigma_{1}^{(n)})-\mathcal{T}^{(n)}(u_{n}|\sigma^{(n)})|,\ |\mathcal{T}_{1}^{(n)}(v_{n}|\sigma_{1}^{(n)})-\mathcal{T}^{(n)}(v_{n}|\sigma^{(n)})|\}.\nonumber 
\end{align}

If $v_{n}\leqslant t_{n},$ then 
\[
\mathcal{T}^{(n)}(u_{n}|\sigma^{(n)})=\mathcal{T}_{2}^{(n)}(u_{n}|\sigma_{2}^{(n)}),\ \mathcal{T}^{(n)}(v_{n}|\sigma^{(n)})=\mathcal{T}_{2}^{(n)}(v_{n}|\sigma_{2}^{(n)}).
\]
It follows from (\ref{gamma}) and (\ref{comparison}) that
\[
\sup_{u_{n}\leqslant t\leqslant v_{n}}|\mathcal{T}_{1}^{(n)}(t|\sigma_{1}^{(n)})-\mathcal{T}^{(n)}(t|\sigma^{(n)})|<2\varepsilon,\ \ \ n>n_{0}.
\]

If $u_{n}\geqslant t_{n},$ since the last two points of $\mathcal{T}_{1}^{(n)}$
are identical (denoted by $x^{(n)}$), it follows that
\begin{align*}
 & \sup_{u_{n}\leqslant t\leqslant v_{n}}|\mathcal{T}_{1}^{(n)}(t|\sigma_{1}^{(n)})-\mathcal{T}^{(n)}(t|\sigma^{(n)})|\\
\leqslant & \max\{|x^{(n)}-\mathcal{T}^{(n)}(u_{n}|\sigma^{(n)}),\ |x^{(n)}-\mathcal{T}^{(n)}(v_{n}|\sigma^{(n)})||\}.
\end{align*}
Obviously 
\[
\mathcal{T}^{(n)}(u_{n}|\sigma^{(n)})=\mathcal{T}_{2}^{(n)}(u_{n}|\sigma_{2}^{(n)}).
\]
But it may not be true for $v_{n}$ since it is possible that $v_{n}=1.$
However, since $\mathcal{T}^{(n)}\prec\mathcal{T}_{2}^{(n)},$ there
exists some $w_{n}>u_{n},$ such that 
\[
\mathcal{T}^{(n)}(v_{n}|\sigma^{(n)})=\mathcal{T}_{2}^{(n)}(w_{n}|\sigma_{2}^{(n)}),\ \ \ (w_{n}=v_{n}\mbox{ if \ensuremath{v_{n}<1}}).
\]
Due to the fact that $\mathcal{T}_{1}^{(n)}\equiv x^{(n)}$ on $[t_{n},1],$
we arrive again at 
\[
\sup_{u_{n}\leqslant t\leqslant v_{n}}|\mathcal{T}_{1}^{(n)}(t|\sigma_{1}^{(n)})-\mathcal{T}^{(n)}(t|\sigma^{(n)})|<2\varepsilon,\ \ \ n>n_{0}.
\]

Consequently, 
\[
\sup_{0\leqslant t\leqslant1}|\mathcal{T}_{1}^{(n)}(t|\sigma_{1}^{(n)})-\mathcal{T}^{(n)}(t|\sigma^{(n)})|<2\varepsilon,\ \ \ n>n_{0}.
\]
It follows that 
\[
\lim_{n\rightarrow\infty}\sup_{0\leqslant t\leqslant1}|\mathcal{T}^{(n)}(t|\sigma^{(n)})-\gamma_{t}|=0,
\]
and in particular, $\mathcal{T}^{(n)}$ converges in trajectory to
$(\gamma_{t})_{0\leqslant t\leqslant1}.$

Now the proof is complete.
\end{proof}

In order to apply Proposition \ref{prop 3.3}, we are going to modify
$\widetilde{\mathcal{X}}^{\varepsilon}$ and choose a suitable parametrization
based on the one for $\widetilde{\mathcal{X}}^{\varepsilon}$ specified
in Subsection 3.1, which is chosen according to the successive visit
time of larger boxes for the diffusion paths (excluding revisit of
the same box before visiting other boxes), so that the assumptions
of Proposition \ref{prop 3.3} are all verified.

The method is the following. By using the notation in Section 3.1,
if $(\zeta_{k}^{\varepsilon},\tau_{l}^{\varepsilon},\zeta_{k+1}^{\varepsilon})$
is such that 
\[
\zeta_{k}^{\varepsilon}<\tau_{l}^{\varepsilon}<\zeta_{k+1}^{\varepsilon}\leqslant1,
\]
then we modify the linear path $\widetilde{X}^{\varepsilon}$ on $[\zeta_{k}^{\varepsilon},\zeta_{k+1}^{\varepsilon}]$
to a new path such that it does not move during $[\zeta_{k}^{\varepsilon},\tau_{l}^{\varepsilon}]$
and goes directly from its initial position at $t=\tau_{l}^{\varepsilon}$
to $\widetilde{X}_{\zeta_{k+1}^{\varepsilon}}^{\varepsilon}$ at $t=\zeta_{k+1}^{\varepsilon}$
with constant velocity. If $ $ 
\[
\zeta_{k}^{\varepsilon}<\tau_{l}^{\varepsilon}<1<\zeta_{k+1}^{\varepsilon},
\]
then we modify the linear path $\widetilde{X}^{\varepsilon}$ on $[\zeta_{k}^{\varepsilon},1]$
(in fact, $\widetilde{X}^{\varepsilon}$ remains still on $[\zeta_{k}^{\varepsilon},1]$)
to a path such that during $[\zeta_{k}^{\varepsilon},\tau_{l}^{\varepsilon}]$
and $[\tau_{l}^{\varepsilon},1]$ it remains still (equals $\widetilde{X}_{\zeta_{k}^{\varepsilon}}^{\varepsilon}$).
It seems that such modification is trivial and does not change anything,
but it does make a slight difference if we are using the associated
P.L.T.. Let $\widehat{X}^{\varepsilon}$ be the modified piecewise
linear path of $\widetilde{X}^{\varepsilon}$ and let $\widehat{\mathcal{X}}^{\varepsilon}$
be the associated P.L.T. of $\widehat{X}^{\varepsilon}.$ If we can
prove that $\widehat{X}^{\varepsilon_{n}}$ converges uniformly to
$(X_{t})_{0\leqslant t\leqslant1}$ with probability one, then all
the assumptions in Proposition \ref{prop 3.3} for the triple sequence
$\{(\mathcal{X}^{\varepsilon_{n}},\varepsilon_{n}\cdot\mathcal{Y}^{\varepsilon_{n}},\widehat{\mathcal{X}}^{\varepsilon_{n}})\}$
are verified, and we will complete the proof of Theorem \ref{thm 2.1}.
In fact, it is just a simple modification of the arguments in Subsection
3.1.
\begin{lem}
With probability one, $(\widehat{X}_{t}^{\varepsilon_{n}})_{0\leqslant t\leqslant1}$
converges uniformly to the diffusion paths $(X_{t})_{0\leqslant t\leqslant1}.$\end{lem}
\begin{proof}
As in the proof of Proposition \ref{prop 3.1}, we need to estimate
$P(\sup_{0\leqslant t\leqslant1}|\widehat{X}_{t}^{\varepsilon}-X_{t}|>\lambda\varepsilon)$
for some universal constant $\lambda$, which reduces to the estimation
of $P(\sup_{\zeta_{j-1}^{\varepsilon}\leqslant t\leqslant\zeta_{j}^{\varepsilon}}|\widehat{X}_{t}^{\varepsilon}-X_{t}|>\lambda\varepsilon,\zeta_{j}^{\varepsilon}\leqslant1),$
$ $ $P(\sup_{\zeta_{l}^{\varepsilon}\leqslant t\leqslant1}|\widehat{X}_{t}^{\varepsilon}-X_{t}|>\lambda\varepsilon,\ M_{V}^{\varepsilon}=l)$
and $P(M_{V}^{\varepsilon}>k)$. 

For the first quantity, from the definition of $\widehat{X}^{\varepsilon}$
we have 
\[
\widehat{X}^{\varepsilon}([\zeta_{j-1}^{\varepsilon},\zeta_{j}^{\varepsilon}])=\widetilde{X}^{\varepsilon}([\zeta_{j-1}^{\varepsilon},\zeta_{j}^{\varepsilon}])
\]
on $\{\zeta_{j}^{\varepsilon}\leqslant1\},$ regardless of whether
the path has visited the smaller box $H_{\boldsymbol{n}_{j-1}^{\varepsilon}}^{\varepsilon}$
during $[\zeta_{j-1}^{\varepsilon},\zeta_{j}^{\varepsilon}]$. Therefore,
the event $\{\sup_{\zeta_{j-1}^{\varepsilon}\leqslant t\leqslant\zeta_{j}^{\varepsilon}}|\widehat{X}_{t}^{\varepsilon}-X_{t}|>\lambda\varepsilon,\zeta_{j}^{\varepsilon}\leqslant1\}$
again implies that during $[\zeta_{j-1}^{\varepsilon},\zeta_{j}^{\varepsilon}]$,
the path must have traveled through many narrow tunnels and spread
far away from the box $V_{\boldsymbol{n}_{j-1}^{\varepsilon}}^{\varepsilon}$
before visiting another box. More precisely, again we have
\[
\{\sup_{\zeta_{j-1}^{\varepsilon}\leqslant t\leqslant\zeta_{j}^{\varepsilon}}|\widehat{X}_{t}^{\varepsilon}-X_{t}|>\lambda\varepsilon,\zeta_{j}^{\varepsilon}\leqslant1\}\subset\{\sigma_{0}<\sigma_{1}<\cdots<\sigma_{\lambda/12}<\infty\},
\]
the same as in the proof of Proposition \ref{prop 3.1}. Similar arguments
apply to the estimation of the second quantity, and the third quantity
has nothing to do with the polygonal approximation. 

Therefore, we can apply exactly the same arguments as in the proof
of Proposition \ref{prop 3.1} to concluded that 
\[
\sum_{n=1}^{\infty}P(\sup_{0\leqslant t\leqslant1}|\widehat{X}_{t}^{\varepsilon}-X_{t}|>\lambda\varepsilon_{n})<\infty,
\]
where $\lambda$ is the universal constant chosen in that proof.
\end{proof}

Now the proof of Theorem \ref{thm 2.1} is complete. 

\begin{rem}
From the proof of Theorem \ref{thm 2.1}, it is not hard to see that the global assumption (C) on the generating vector fields can be weakened to a local one to some extend. In fact, the only property of the vector fields we've used from Assumption (C) is that at every point on the boundary of $H_z^{\varepsilon}$, the vector fields ${V_1,\cdots,V_d}$ do not generate a subspace of the tangent space at that point. Therefore, it suffices to assume that for each $z$ and $\varepsilon$, there exists a small rotation $O$ (an orthogonal transformation) such that after rotating the box $H_z^{\varepsilon}$ by $O$ with respect to its center, the vector fields do not generate a subspace of the tangent space at every point on the boundary. The smallness of the rotation $O$ can be quantified as follows. If we let 
\[
\widetilde{H}_z^{\varepsilon}=\varepsilon z+O(H_z^{\varepsilon}-\varepsilon z)
\] be the rotated box, then $O$ should satisfy the condition that for any $x=(x_1,\cdots,x_N)\in\widetilde{H}_z^{\varepsilon}$, 
\[
|x_i-\varepsilon z_i|<\frac{\varepsilon}{2},\ \forall i=1,\cdots,N.
\] 
This is to ensure that the geometric configuration, in particular the tunnel structure, is not damaged, so that the whole proof of Theorem \ref{thm 2.1} carries through in the same way.
\end{rem}

\section*{Acknowledgement}
The authors wish to thank the referees for their very careful reading on the manuscript and very useful suggestions. The authors are supported by the Oxford-Man Institute at University of Oxford.


\begin{thebibliography}{99}

\bibitem{ben1984poisson}G. Ben Arous, S. Kusuoka, D.W. Stroock, The Poisson kernel for certain degenerate elliptic operators, \textit{J. Funct. Anal.} 56 (2) (1984) 171-209.

\bibitem{chen1977iterated}K.T. Chen, Iterated path integrals, \textit{Bull. Amer. Math. Soc.} 83 (5) (1977) 831-879.

\bibitem{hambly2010uniqueness}B. Hambly, T. Lyons, Uniqueness for the signature of a path of bounded variation and the reduced path group, \textit{Ann. of Math.} 171(1) (2010) 109-167.

\bibitem{hormander1967hypoelliptic}L. H{\"o}rmander, Hypoelliptic second order differential equations, \textit{Acta Math.} 119 (1) (1967) 147-171.

\bibitem{ikeda1989stochastic}N. Ikeda, S. Watanabe, \textit{Stochastic Differential Equations and Diffusion Processes}, North-Holland, 1989.

\bibitem{lejan2012stratonovich}Y. Le Jan, Z. Qian, Stratonovich's signatures of Brownian motion determine Brownian sample paths,  \textit{Probab. Theory Relat. Fields} 157 (2012) 440-454.

\bibitem{lyons1998differential}T. Lyons, Differential equations driven by rough signals, \textit{Rev. Mat. Iberoamericana} 14 (2) (1998) 215-310.

\bibitem{lyons2007differential}T. Lyons, M. Caruana, T. L{\'e}vy, J. Picard, \textit{Differential Equations Driven by Rough Paths}, \textit{$\acute{E}$cole d'$\acute{e}$t$\acute{e}$ de probabilit$\acute{e}$s de Saint-Flour XXXIV-2004}, Springer, 2007.

\bibitem{lyons2003system}T. Lyons, Z. Qian, \textit{System Control and Rough Paths}, Oxford University Press, 2002.

\bibitem{wong1966relationship}E. Wong, M. Zakai, On the relation between ordinary and stochastic differential equations, \textit{Internat. J. Engrg. Sci.} 3 (2) (1965) 213-229.

\end{thebibliography}
\end{document}